\documentclass[12pt]{article}
\pdfoutput=1
\usepackage[utf8]{inputenc}
\usepackage{amsfonts,amsmath,amsthm,amssymb}
\usepackage{nicefrac}
\usepackage{enumitem}
\usepackage[colorlinks,allcolors=blue]{hyperref}
\usepackage{url}
\usepackage[longnamesfirst, authoryear]{natbib}
\usepackage[title]{appendix}
\usepackage[margin=1.2in]{geometry}

\newtheorem{theorem}{Theorem}

\newtheorem{lemma}{Lemma}

\newtheorem{proposition}{Proposition}

\newtheorem{corollary}{Corollary}

\theoremstyle{remark}

\newtheorem{definition}{Definition}

\newtheorem{fact}{Fact}

\let\Pr\relax
\DeclareMathOperator\Pr{\mathbb{P}}
\DeclareMathOperator\Qr{\mathbb{Q}}
\DeclareMathOperator\E{\mathbb{E}}

\newcommand{\KL}[2]{ { D \left({#1} \;\middle\Vert\; {#2}\right) } }

\newcommand{\KLb}[2]{ { D_2 \left({#1} \;\middle\Vert\; {#2}\right) } }

\newcommand{\bw}{{\pmb{w}}}
\newcommand{\bmu}{{\pmb{\mu}}}

\newcommand{\btheta}{{\pmb{\theta}}}
\newcommand{\bTheta}{{\pmb{\Theta}}}
\newcommand{\blambda}{{\pmb{\lambda}}}

\newcommand{\Atheta}{{\text{Alt}(\pmb{\theta})}}
\newcommand{\Prtheta}{\Pr_\btheta^{\calA_\delta}}
\newcommand{\Etheta}{\E_\btheta^{\calA_\delta}}
\newcommand{\Prlambda}{\Pr_\blambda^{\calA_\delta}}
\newcommand{\Elambda}{\E_\blambda^{\calA_\delta}}

\newcommand{\calA}{{\mathcal A}}

\newcommand{\calE}{{\mathcal E}}
\newcommand{\calF}{{\mathcal F}}

\newcommand{\calM}{{\mathcal M}}

\newcommand{\Intp}[1]{ { {\mathbb Z_{\ge 0}}^{#1} } }
\newcommand{\Intpp}[1]{ { {\mathbb Z_{> 0}}^{#1} } }
\newcommand{\Real}[1]{ { {\mathbb R}^{#1} } }
\newcommand{\Realp}[1]{ { {\mathbb R}_{\ge 0}^{#1} } }

\newcommand{\Simplex}[1]{ { {\mathcal M}_1 \left({#1}\right) } }

\newcommand{\dss}{\displaystyle}

\newcommand{\inv}{^{-1}}

\DeclareMathOperator*{\argmax}{arg\,max}
\DeclareMathOperator*{\argmin}{arg\,min}

\title{\textbf{Optimal Best Markovian Arm Identification with Fixed Confidence}}
\author{
Vrettos Moulos
\thanks{Supported in part by the NSF grant CCF-1816861.}
\\
University of California Berkeley \\
\href{mailto:vrettos@berkeley.edu}{vrettos@berkeley.edu}
}
\date{}

\begin{document}

\maketitle

\begin{abstract}
We give a complete characterization of the sampling complexity
of best Markovian arm identification in one-parameter Markovian bandit models. We derive instance specific nonasymptotic and asymptotic lower bounds which generalize those of the IID setting.
We analyze the Track-and-Stop strategy, initially proposed for the IID setting, and we prove that asymptotically it is at most a factor of four apart from the lower bound. Our one-parameter Markovian bandit model is based on the notion of an exponential family of stochastic matrices for which we establish many useful properties. For the analysis of the Track-and-Stop strategy we derive a novel concentration inequality for Markov chains that may be of interest in its own right.
\end{abstract}

\section{Introduction}\label{sec:intro}

This paper is about optimal best Markovian arm identification with fixed confidence. 
There are $K$ independent options which are referred to as arms. 
Each arm $a$ is associated with a discrete time stochastic process, which is characterized by a parameter $\theta_a$ and it's governed by the probability law $\Pr_{\theta_a}$.
At each round we select one arm, without any prior knowledge of the statistics of the stochastic processes.
The stochastic process that corresponds to the selected arm evolves by one time step, and we observe this evolution through a reward function, while the stochastic processes for the rest of the arms stay still.
A confidence level $\delta \in (0, 1)$ is prescribed, and our goal is to identify the arm that corresponds to the process with the highest stationary mean with probability at least $1-\delta$, and using as few samples as possible.

\subsection{Contributions}

In the work of~\cite{GK16} the discrete time stochastic process associated with each arm $a$ is assumed to be an IID process.
Here we go one step further and we study more complicated dependent processes, which allow us to use more expressive models in the stochastic multi-armed bandits framework. More specifically we consider the case that each $\Pr_{\theta_a}$ is the law of an irreducible finite state Markov chain associated with a stationary mean $\mu (\theta_a)$.
We establish a lower bound (\autoref{thm:lower-bound}) for the expected sample complexity,
as well as an analysis of the Track-and-Stop strategy,
proposed for the IID setting in~\cite{GK16},
which shows (\autoref{thm:upper-bound}) that asymptotically the Track-and-Stop strategy in the Markovian dependence setting attains a sample complexity which is at most a factor of four apart from our asymptotic lower bound. Both our lower and upper bounds extend the work of~\cite{GK16} in the more complicated and more general Markovian dependence setting.

The abstract framework of multi-armed bandits has numerous applications in areas like clinical trials, ad placement, adaptive routing, resource allocation, gambling etc. For more context we refer the interested reader to the survey of~\cite{BC12}. Here we generalize this model to allow for the presence of Markovian dependence, enabling this way the practitioner to use richer and more expressive models for the various applications.
In particular, Markovian dependence allows models where the distribution of next sample depends on the sample just observed.
This way one can model for instance the evolution of a rigged slot machine,
which as soon as it generates a big reward for the gambler, it changes the reward distribution to a distribution which is skewed towards smaller rewards.

Our key technical contributions stem from the large deviations theory for Markov chains~\cite{Miller61, Donsker-Varadhan-I-75, Ellis84, Dembo-Zeitouni-98}. In particular we utilize the concept
of an \emph{exponential family of stochastic matrices}, first introduced in~\cite{Miller61}, in order to model our one-parameter Markovian bandit model. Many properties of the family are established which are then used for our analysis of the Track-and-Stop strategy. The most important one is an optimal concentration inequality for the empirical means of Markov chains
(\autoref{thm:concentration-bound}). We are able to establish this inequality for a large class of Markov chains, including those that all the transitions have positive probability.
Prior work on the topic,~\cite{Gillman93, Dinwoodie95, Lezaud98, LP04},
fails to capture the optimal exponential decay, or introduces a polynomial prefactor,~\cite{Davisson-Longo-Sgarro-81}, as opposed to our constant prefactor. This result may be of independent interest due to the wide applicability of Markov chains in many aspects of learning theory such as various aspects of reinforcement learning, Markov chain Monte Carlo and others. 

\subsection{Related Work}

The cornerstone of stochastic multi-armed bandits is the seminal work of~\cite{Lai-Robbins-85}.
They considered $K$ IID process with the objective being to maximize the expected value of the sum of the observed rewards, or equivalently to minimize the so called \emph{regret}. In the same spirit~\cite{Ananth-Varaiya-Walrand-I-87,Ananth-Varaiya-Walrand-II-87} examine the generalization where one is allowed to collect multiple rewards at each time step, first in the case that processes are IID~\cite{Ananth-Varaiya-Walrand-I-87}, and then in the case that the processes are irreducible and aperiodic Markov chains~\cite{Ananth-Varaiya-Walrand-I-87}. A survey of the regret minimization literature is contained in~\cite{BC12}.

An alternative objective is the one of identifying the process with the highest stationary mean as fast as and as accurately as possible, notions which are made precise in~\autoref{sec:fam}.
In the IID setting, \cite{EMM06} establish an elimination based algorithm in order to find an approximate best arm, and~\cite{MT04} provide a matching lower bound. \cite{JMNB14} propose an upper confidence strategy, inspired by the law of iterated logarithm, for exact best arm identification given some fixed level of confidence. In the asymptotic high confidence regime, the problem is settled by the work of~\cite{GK16},
who provide instance specific matching lower and upper bounds. For their upper bound they propose the Track-and-Stop strategy which is further explored in the work of~\cite{KK18}.

The earliest reference for the exponential family of stochastic matrices which is being used to model the Markovian arms can be found in the work of~\cite{Miller61}. 
Exponential families of stochastic matrices lie in the heart of the theory of large deviations for Markov processes, which was popularized with the pioneering work of~\cite{Donsker-Varadhan-I-75}.
A comprehensive overview of the theory can be found in the book~\cite{Dembo-Zeitouni-98}.
Naturally they also show up when one conditions on the second order empirical distribution of a Markov chain, see the work of~\cite{Csizar-Cover-Choi-87} about conditional limit theorems.
A variant of the exponential family that we are going to discuss has been developed in the context of hypothesis testing in~\cite{Nakagawa-Kanaya-93}.
A more recent development by~\cite{Nagaoka-05}
gives an information geometry perspective to this concept, and the work~\cite{HW16} examines parameter estimation for the exponential family.
Our development of the exponential family of stochastic matrices tries to parallel the development of simple exponential families of probability distributions of~\cite{Wainwright-Jordan-08}.

Regarding concentration inequalities for Markov chains one of the earliest works~\cite{Davisson-Longo-Sgarro-81} is based on counting, and is able to capture the optimal rate of exponential decay dictated by the theory of large deviations, but has a suboptimal polynomial prefactor. More recent approaches follow the line of work started by~\cite{Gillman93}, who used matrix perturbation theory to derive a bound for reversible Markov chains. This bound attains a constant prefactor but with a suboptimal rate of exponential decay which depends on the spectral gap of the transition matrix. This work was later extended by~\cite{Dinwoodie95,Lezaud98} but still with a sub-optimal rate. The work of~\cite{LP04} reduces the problem to a two state Markov chain, and attains the optimal rate only for the case of a two state Markov chain. 
\cite{CLLM12} obtain rates that depend on the mixing time of the chain rather than the spectral gap, but which are still suboptimal.

\section{Problem Formulation}

\subsection{One-parameter family of Markov Chains}\label{sec:fam}

In order to model the problem we will use a one-parameter family of Markov chains on a finite state space $S$.
Each Markov chain in the family corresponds to a parameter $\theta \in \Theta$, where $\Theta \subseteq \Real{}$ is the parameter space,
and is completely characterized by an initial distribution
$q_\theta = [q_\theta (x)]_{x \in S}$, and a stochastic transition matrix $P_\theta = [P_\theta (x, y)]_{x, y \in S}$, which satisfy the following conditions.
\begin{gather}
P_\theta ~\text{is irreducible for all}~ \theta \in \Theta. \label{eqn:irred} \\
P_\theta (x, y) > 0 \;\Rightarrow\; P_\lambda (x, y) > 0, ~\text{for all}~ \theta, \lambda \in \Theta, ~ x, y \in S. \label{eqn:supp} \\
q_\theta (x) > 0 \;\Rightarrow\; q_\lambda (x) > 0, ~\text{for all}~ \theta, \lambda \in \Theta, ~ x \in S. \label{eqn:init}
\end{gather}

There are $K$ Markovian arms with parameters $\btheta = (\theta_1, \ldots, \theta_K) \in \Theta^K$, and each arm $a \in [K] = \{1, \ldots, K\}$ evolves as a Markov chain with parameter $\theta_a$ which we denote by $\{X_n^a\}_{x \in \Intp{}}$.
A non-constant real valued reward function $f : S \to \Real{}$ is applied at each state and produces the reward process $\{Y_n^a\}_{n \in \Intp{}}$ given by $Y_n^a = f (X_n^a)$. We can only observe the reward process but not the internal Markov chain. Note that the reward process is a function of the Markov chain and so in general it will have more complicated dependencies than the Markov chain. The reward process is a Markov chain if and only if $f$ is injective. For each $\theta \in \Theta$ there is a unique stationary distribution $\pi_\theta = [\pi_\theta (x)]_{x \in S}$
associated with the stochastic matrix $P_\theta$, due to~\eqref{eqn:irred}. This allows us to define the stationary reward of the Markov chain corresponding to the parameter $\theta$ as $\mu (\theta) = \sum_x f (x) \pi_\theta (x)$. We will assume that among the $K$ Markovian arms there exists precisely one that possess the highest stationary mean, and we will denote this arm by $a^* (\btheta)$, so in particular
\[
\{a^* (\btheta)\} = \argmax_{a \in [K]} \mu (\theta_a).
\]
The set of all parameter configurations that possess a unique highest mean is denoted by
\[
\bTheta = \left\{
\btheta \in \Theta^K :
\left|\argmax_{a \in [K]} \mu (\theta_a)\right| = 1
\right\}.
\]
The \emph{Kullback-Leibler divergence rate} characterizes the sample complexity of the Markovian identification problem that we are about to study. For two Markov chains of the one-parameter family that are indexed by $\theta$ and $\lambda$ respectively it is given by,
\[
\KL{\theta}{\lambda} = 
\sum_{x, y \in S} \log \frac{P_\theta (x, y)}{P_\lambda (x, y)}
\pi_\theta (x) P_\theta (x, y),
\]
where we use the standard notational conventions $\log 0 = \infty, ~ \log \frac{\alpha}{0} = \infty ~ \text{if} ~ \alpha > 0$,
and $0 \log 0 = 0 \ln \frac{0}{0} = 0$.
It is always nonnegative, $\KL{\theta}{\lambda} \ge 0$, with equality occurring if and only if $P_\theta = P_\lambda$, and so $\mu (\theta) \neq \mu (\lambda)$ yields that $\KL{\theta}{\lambda} > 0$.
Furthermore, $\KL{\theta}{\lambda} < \infty$ due to~\eqref{eqn:supp}.

With some abuse of notation we will also write $\KL{\Pr}{\Qr}$ for the
Kullback-Leibler divergence between two probability measures $\Pr$ and $\Qr$ on the same measurable space, which is defined as
\[
\KL{\Pr}{\Qr} = 
\begin{cases}
\E_{\Pr} \left[\log \frac{d \Pr}{d \Qr}\right], & \text{if}~ \Pr \ll \Qr \\
\infty, & \text{otherwise},
\end{cases}
\]
where $\Pr \ll \Qr$ means that $\Pr$ is absolutely continuous with respect to $\Qr$, and in that case $\frac{d \Pr}{d \Qr}$ denotes the Radon-Nikodym derivative of $\Pr$ with respect to $\Qr$.

\subsection{Best Markovian Arm Identification with Fixed Confidence}

Let $\btheta \in \bTheta$ be an unknown parameter configuration for the $K$ Markovian arms.
Let $\delta \in (0, 1)$ be a given confidence level.
Our goal is to identify $a^* (\btheta)$ with probability at least $1-\delta$ using as few samples as possible. 
At each time $t$ we select a single arm $A_t$ and we observe the next sample from the reward process $\{Y_n^{A_t}\}_{n \in \Intp{}}$,
while all the other reward processes stay still.
Let $N_a (t) = \sum_{s=1}^t I_{\{A_s = a\}} - 1$ be the number of transitions of the Markovian arm $a$ up to time $t$.
Let $\calF_t$ be the $\sigma$-field generated by our choices $A_1, \ldots, A_t$ and the observations $\{Y_n^1\}_{n=0}^{N_1 (t)}, \ldots, \{Y_n^a\}_{n=0}^{N_K (t)}$.
A sampling strategy, $\calA_\delta$,
is a triple $\calA_\delta = ((A_t)_{t \in \Intpp{}}, \tau_\delta, \hat{a}_{\tau_\delta})$
consisting of:
\begin{itemize}
    \item a \emph{sampling rule} $(A_t)_{t \in \Intpp{}}$, which based on the past decisions and observations
    $\calF_t$, determines which arm $A_{t+1}$ we should sample next, so $A_{t+1}$ is $\calF_t$-measurable;
    
    \item a \emph{stopping rule} $\tau_\delta$, which denotes the end of the data collection phase and is a stopping time with respect to the filtration $(\calF_t)_{t \in \Intpp{}}$, such that $\Elambda [\tau_\delta] < \infty$
    for all $\blambda \in \bTheta$;
    
    \item a \emph{decision rule} $\hat{a}_{\tau_\delta}$, which is $\calF_{\tau_\delta}$-measurable, and determines the arm that we estimate to be the best one.
    
\end{itemize}

Sampling strategies need to perform well across all possible parameter configurations in $\bTheta$, therefore we need to restrict our strategies to a class of \emph{uniformly accurate} strategies. This motivates the following standard definition.
\begin{definition}[$\delta$-PC]
Given a confidence level $\delta \in (0, 1)$, a sampling strategy $\calA_\delta = ((A_t)_{t \in \Intpp{}}, \tau_\delta, \hat{a}_{\tau_\delta})$ is called $\delta$-PC (Probably Correct) if,
\[
\Prlambda (\hat{a}_{\tau_\delta} \neq a^* (\blambda)) \le \delta, ~\text{for all}~
\blambda \in \bTheta.
\]
\end{definition}
Therefore our goal is to study the quantity,
\[
\inf_{\calA_\delta : \delta-PC} \Etheta [\tau_\delta],
\]
both in terms of finding a lower bound, i.e. establishing that no $\delta$-PC strategy can have expected sample complexity less than our lower bound, and also in terms of finding an upper bound, i.e. a $\delta$-PC strategy with very small expected sample complexity. We will do so in the high confidence regime of $\delta \to 0$, by establishing instance specific lower and upper bounds which differ just by a factor of four. 
\section{Lower Bound on the Sample Complexity}

Deriving lower bounds in the multi-armed bandits setting is a task performed by change of measure arguments 
initial introduced by~\cite{Lai-Robbins-85}.
Those change of measure arguments capture the simple idea
that in order to identify the best arm we should at least be able to differentiate between two bandit models that exhibit different best arms but are statistically similar.
Fix $\theta \in \bTheta$, and define the set of parameter configurations that exhibit as best arm
an arm different than $a^* (\btheta)$ by
\[
\Atheta = \{\blambda \in \bTheta : a^* (\blambda) \neq a^* (\btheta)\}.
\]
Then we consider an alternative parametrization $\blambda \in \Atheta$ and we write their log-likelihood ratio up to time $t$
\begin{equation}\label{eqn:log-ratio}
\begin{aligned}
& \log \left(\frac{d \Prtheta \mid_{\calF_t}}{d \Prlambda \mid_{\calF_t}}\right) = 
\sum_{a=1}^K I_{\{N_a (t) \ge 0 \}} \log \frac{q_{\theta_a} (X_0^a)}{q_{\lambda_a} (X_0^a)}  \\
&\qquad +
\sum_{a=1}^K  \sum_{x, y}
N_a (x, y, 0, t)
\log \frac{P_{\theta_a} (x, y)}{P_{\lambda_a} (x, y)},
\end{aligned}
\end{equation}
where $N_a (x, y, 0, t) = \sum_{s=0}^{t-1} 1 \{X_s^a = x, X_{s+1}^a = y\}$. The log-likelihood ratio enables us to perform changes of measure
for fixed times $t$, and more generally for stopping times $\tau$ with respect to $(\calF_t)_{t \in \Intpp{}}$, which are $\Pr_\btheta^{\calA_\delta}$ and $\Pr_\blambda^{\calA_\delta}$-a.s. finite, through the following change of measure formula,
\begin{equation}\label{eqn:change-measure}
\Prlambda (\calE) = 
\Etheta \left[
I_\calE \frac{d \Pr_\blambda \mid_{\calF_{\tau}}}{d \Pr_\btheta \mid_{\calF_{\tau}}}
\right], ~\text{for any}~ \calE \in \calF_\tau.
\end{equation}
In order to derive our lower bound we use a technique developed
for the IID case by~\cite{GK16} which combines several changes of measure at once.
To make this technique work in the Markovian setting we need the following inequality which we derive in~\autoref{app:lower-bound} 
using a renewal argument for Markov chains.
\begin{lemma}\label{lem:KL-bound}
Let $\btheta \in \bTheta$ and 
$\blambda \in \Atheta$ be two parameter configurations. 
Let $\tau$ be a stopping time with respect to $(\calF_t)_{t \in \Intpp{}}$, with $\Etheta [\tau], ~ \Elambda [\tau] < \infty$.
Then
\begin{align*}
& \KL{\Prtheta \mid_{\calF_\tau}}{\Prlambda \mid_{\calF_\tau}}  \\
& \qquad \le
\sum_{a=1}^K \KL{q_{\theta_a}}{q_{\lambda_a}} +
\sum_{a=1}^K \left(\Etheta [N_a (\tau)] + R_{\theta_a}\right) \KL{\theta_a}{\lambda_a},
\end{align*}
where $R_{\theta_a} = \E_{\theta_a}\: \left[\inf \{n > 0 : X_n^a = X_ 0^a\}\right] < \infty$, the first summand is finite due to~\eqref{eqn:init}, and the second summand is finite due to~\eqref{eqn:supp}.
\end{lemma}
Combining those ingredients with the data processing inequality
we derive our instance specific lower bound for the Markovian bandit identification problem in~\autoref{app:lower-bound}.
\begin{theorem}\label{thm:lower-bound}
Assume that the one-parameter family of Markov chains on the finite state space $S$ satisfies conditions~\eqref{eqn:irred},~\eqref{eqn:supp}, and~\eqref{eqn:init}.
Fix $\delta \in (0, 1)$, let $f:S \to \Real{}$ be a nonconstant reward function, let $\calA_\delta$ be a $\delta$-PC sampling strategy, and fix a parameter configuration $\btheta \in \bTheta$. Then
\[
T^* (\btheta) \le \liminf_{\delta \to 0}
\frac{\Etheta [\tau_\delta]}{\log \frac{1}{\delta}},
\]
where
\[
T^* (\btheta)\inv =
\sup_{\bw \in \Simplex{[K]}} \inf_{\blambda \in \Atheta}\: \sum_{a=1}^K
w_a \KL{\theta_a}{\lambda_a},
\]
and $\Simplex{[K]}$ denotes the set of all probability distributions on $[K]$.
\end{theorem}
As noted in~\cite{GK16} the $\sup$ in the definition of $T^* (\btheta)$ is actually attained uniquely, and therefore we can define $\bw^* (\btheta)$ as the unique maximizer,
\[
\{\bw^* (\btheta)\} = \argmax_{\bw \in \Simplex{[K]}} \inf_{\blambda \in \Atheta}
\sum_{a=1}^K w_a \KL{\theta_a}{\lambda_a}.
\]
\section{One-Parameter Exponential Family of Markov Chains}

\subsection{Definition and Basic Properties}
\label{sec:exp-fam}

In this section we instantiate the abstract one-parameter family of Markov chains from~\autoref{sec:fam}, with the one-parameter exponential family of Markov chains.
Given the finite state space $S$, and the nonconstant reward function $f : S \to \Real{}$, we define $M = \max_x f (x)$ and $m = \min_x f (x)$.
Based on $f$ we construct two subsets of state space,
$S_M = \{x \in S : f (x) = M\}$ and
$S_m = \{x \in S : f (x) = m\}$, corresponding to states of maximum and minimum $f$-value respectively.
Our goal is to create a family of Markov chains which can realize any stationary mean in the interval $(m, M)$, which will be later used in order to model the Markovian arms. Towards this goal we use as a \emph{generator} for our family, an irreducible stochastic matrix $P$ which satisfies the following properties.
\begin{align}
& \text{The submatrix of $P$ with rows and columns in $S_M$ is irreducible.}\label{eqn:as-up-1} \\
& \text{For every $x \in S - S_M$, there is a
$y \in S_M$ such that $P (x, y) > 0$.}\label{eqn:as-up-2} \\
& \text{The submatrix of $P$ with rows and columns in $S_m$ is irreducible.}\label{eqn:as-low-1} \\
& \text{For every $x \in S - S_m$, there is a
$y \in S_m$ such that $P (x, y) > 0$.}\label{eqn:as-low-2}
\end{align}
For example, a positive stochastic matrix, i.e. one where all the transition probabilities are positive, satisfies all those properties. Note that in practice this can always be attained by substituting zero transition probabilities with $\epsilon$ transition probabilities, where $\epsilon \in (0, 1)$ is some small constant.

Our parameter space will be the whole real line, $\Theta = \Real{}$. Given a parameter $\theta \in \Theta$, we pick an arbitrary initial distribution $q_\theta \in \Simplex{S}$ such that $q_\theta (x) > 0$ for all $x \in S$, and we tilt exponentially all the the transitions of $P$ by constructing the matrix $\Tilde{P}_\theta (x, y) = P (x, y) e^{\theta f (y)}$.
Note that $\Tilde{P}_\theta$ is not a stochastic matrix,
but we can normalize it and turn it into a stochastic matrix by invoking the Perron-Frobenius theory.
Let $\rho (\theta)$ be the spectral radius of $\tilde{P}_\theta$.
From the Perron-Frobenius theory we know that $\rho (\theta)$
is a simple eigenvalue of $\tilde{P}_\theta$, called the Perron-Frobenius eigenvalue, associated with unique left and right eigenvectors
$u_\theta, ~ v_\theta$
such that they are both positive, $\sum_x u_\theta (x) = 1,$  and $\sum_x u_\theta (x) v_\theta (x) = 1$, see for instance Theorem 8.4.4 in the book~\cite{HJ13}.
Let
$A (\theta) = \log  \rho (\theta)$ be the log-Perron-Frobenius eigenvalue, a quantity which plays a role similar to that of a log-moment-generating function.
From $\tilde{P}_\theta$ we can construct an irreducible nonnegative matrix
\[
P_\theta (x, y) = \frac{\tilde{P}_\theta (x, y) v_\theta (y)}
{\rho (\theta) v_\theta (x)} =
\frac{v_\theta (y)}{v_\theta (x)}
e^{\theta \phi (y) - A (\theta)} P (x, y),
\]
which is stochastic, since
\[
\sum_y P_\theta (x, y) = \frac{1}{\rho (\theta) v_\theta (x)} \cdot \sum_y \tilde{P}_\theta (x, y) v_\theta (y) = 1.
\]
In addition its stationary distributions is given by
\[
\pi_\theta (x) = u_\theta (x) v_\theta (x),
\]
since
\[
\sum_x \pi_\theta (x) P_\theta (x, y) =
\frac{v_\theta (y)}{\rho (\theta)} \cdot 
\sum_x u_\theta (x) \tilde{P}_\theta (x, y) = 
u_\theta (y) v_\theta (y) = \pi_\theta (y).
\]
Note that the generator stochastic matrix $P$, is the member of the family that corresponds to $\theta = 0$, i.e. $P = P_0, \rho (0) = 1$, and $A (0) = 0$.

The following lemma, whose proof is presented in~\autoref{app:exp-fam}, suggests that the family can be reparametrized using the mean parameters $\mu (\theta)$.
More specifically $\dot{A}$ is a strictly increasing bijection between the set $\Theta$ of canonical parameters and the set
$\calM = \{\mu \in (m, M) : \mu (\theta) = \mu, ~ \text{for some} ~ \theta \in \Theta\}$ of mean parameters.
Therefore with some abuse of notation, we will write $u_\mu, v_\mu, P_\mu, \pi_\mu$ for $u_{\dot{A}\inv (\mu)}, v_{\dot{A}\inv (\mu)}, P_{\dot{A}\inv (\mu)}, \pi_{\dot{A}\inv (\mu)}$, and $\KL{\mu_1}{\mu_2}$ for $\KL{\dot{A}\inv (\mu_1)}{\dot{A}\inv (\mu_2)}$.
\begin{lemma}\label{lem:dual-map}
Let $P$ be an irreducible stochastic matrix stochastic matrix on a finite state space $S$ which combined with a real-valued function $f : S \to \Real{}$ satisfies~\eqref{eqn:as-up-1},~\eqref{eqn:as-up-2},~\eqref{eqn:as-low-1} and~\eqref{eqn:as-low-2}. Then the following properties hold true for the exponential family of stochastic matrices generated by $P$ and $f$.
\begin{enumerate}[label=(\alph*)]
    \item $\rho (\theta), ~A (\theta), ~ u_\theta$ and $v_\theta$ are analytic functions of $\theta$ on $\Theta = \Real{}$.
    \item $\dot{A} (\theta) = \mu (\theta)$, for all $\theta \in \Theta$.
    \item $\dot{A} (\theta)$ is strictly increasing.
    \item $\calM = (m, M)$.
\end{enumerate}
\end{lemma}

\subsection{Concentration for Markov Chains}
\label{sec:concentration}

For a Markov chain $\{X_n\}_{n \in \Intp{}}$, driven by an irreducible transition matrix $P$ and an initial distribution $q$, the large deviations theory,~\cite{Miller61, Donsker-Varadhan-I-75, Ellis84, Dembo-Zeitouni-98}, suggests that the probability of the large deviation event
$\{f (X_1) + \ldots + f (X_n) \ge n \mu\}$, when $\mu$ is greater or equal than the stationary mean $\mu (0)$, asymptotically is an exponential decay with the rate of the decay given by a Kullback-Leibler divergence rate. In particular Theorem 3.1.2. from~\cite{Dembo-Zeitouni-98} in our context can be written as
\[
\lim_{n \to \infty} \frac{1}{n} \log
\Pr_0 (f (X_1) + \ldots + f (X_n) \ge n \mu) =
- A^* (\mu), ~ \text{for any} ~ \mu \ge \mu (0),
\]
where $A^* (\mu) = \sup_{\theta \in \Real{}} \{\theta \mu -  A (\theta)\}$ is the convex conjugate of the log-Perron-Frobenius eigenvalue and represents a Kullback-Leibler divergence rate as we illustrate in~\autoref{lem:conv-conj}.

In the following theorem we present a concentration inequality for Markov chains which attains the rate of exponential decay prescribed from the large deviations theory, as well as a constant prefactor which is independent from $\mu$.
\begin{theorem}\label{thm:concentration-bound}
Let $S$ be a finite state space, and let $P$ be an irreducible stochastic matrix on $S$, which combined with a function $f : S \to \Real{}$
satisfies~\eqref{eqn:as-up-1},~\eqref{eqn:as-up-2},~\eqref{eqn:as-low-1}, and~\eqref{eqn:as-low-2}.
Fix $\theta \in \Real{}$, and let $\{X_n\}_{n \in \Intp{}}$ be a Markov chain on $S$, which is driven by $P_\theta$, the stochastic matrix from the exponential family which corresponds to the parameter $\theta$ and has stationary mean $\mu (\theta)$.
Then
\[
\Pr_\theta \left(f (X_1) + \ldots + f (X_n) \ge n \mu\right) \le C^2 e^{-n \KL{\mu}{\mu (\theta)}},
~\text{for}~ \mu \in [\mu (\theta), M],
\]
where $C = C (P, f)$ is a constant depending only on the generator stochastic matrix $P$ and the function $f$.
In particular, if $P$ is a positive stochastic matrix then we can take
$C = \max_{x, y, z} \frac{P (y, z)}{P (x, z)}$.
\end{theorem}
We note that in the special case that the process is an IID process the constant $C (P, \phi)$ can be taken to be $1$,
and thus~\autoref{thm:concentration-bound} generalizes the classic Cramer-Chernoff bound,~\cite{Chernoff52}.
Observe also that~\autoref{thm:concentration-bound} has a straightforward counterpart for the lower tail as well.

Moreover our inequality is optimal up to the constant prefactor, since the exponential decay is unimprovable due to the large deviations theory, while with respect to the prefactor we can not expect anything better than a constant because otherwise we would contradict the central limit theorem for Markov chains.
In particular, when our conditions on $P$ and $f$ are met, our bound dominates similar bounds given by~\cite{Davisson-Longo-Sgarro-81, Gillman93, Dinwoodie95, Lezaud98, LP04}.

We give a proof of~\autoref{thm:concentration-bound} in
~\autoref{app:concentration}, where the main techniques involved
are a uniform upper bound on the ratio of the entries of the right Perron-Frobenius eigenvector, as well as an approximation of the log-Perron-Frobenius eigenvalue using the log-moment-generating function.
\section{Upper Bound on the Sample Complexity: the \texorpdfstring{$\pmb{(\alpha, \delta)}$}{Lg}-Track-and-Stop Strategy}\label{sec:upper-bound}

The $(\alpha, \delta)$-Track-and-Stop strategy, which was proposed in~\cite{GK16} in order to tackle the IID setting, tries to track the optimal weights $w_a^* (\btheta)$. 
In the sequel we will also write $\bw^* (\bmu)$, with $\bmu = (\mu (\theta_1), \ldots, \mu (\theta_K))$,  
to denote $\bw^* (\btheta)$.
Not having access to $\bmu$,
the $(\alpha, \delta)$-Track-and-Stop strategy tries to approximate $\bmu$ using sample means.
Let $\hat{\bmu} (t) = (\hat{\mu}_1 (N_1 (t)), \ldots, \hat{\mu}_K (N_K (t)))$ be the sample means of the $K$ Markov chains when $t$ samples have been observed overall and the calculation of the very first sample from each Markov chain is excluded from the calculation of its sample mean, i.e.
\[
\hat{\mu}_a (t) = \frac{1}{N_a (t)} \sum_{s=1}^{N_a (t)} Y_s^a.
\]
By imposing sufficient \emph{exploration} the law of large numbers for Markov chains will kick in and the sample means $\hat{\bmu} (t)$ will almost surely converge to the true means $\bmu$, as $t \to \infty$.

We proceed by briefly describing the three components of the 
$(\alpha, \delta)$-Track-and-Stop strategy.

\subsection{Sampling Rule: Tracking the Optimal Proportions}

For initialization reasons the first $2 K$ samples that we are going to observe are
$Y_0^1, Y_1^1, \ldots, Y_0^K, Y_1^K$.
After that, for $t \ge 2 K$ we let $U_t = \{a : N_a (t) < \sqrt{t} - K/2\}$ and we follow the tracking rule:
\[
A_{t+1} \in
\begin{cases}
\dss \argmin_{a \in U_t}\: N_a (t), & \text{if} ~ U_t \neq \emptyset \quad \text{(forced exploration)}, \\
\dss \argmax_{a=1,\ldots,K}\: \left\{
w_a^* (\hat{\bmu} (t)) - \frac{N_a (t)}{t}\right\}, & \text{otherwise} \quad \text{(direct tracking)}.
\end{cases}
\]
The forced exploration step is there to ensure that
$\hat{\bmu} (t) \stackrel{a.s.}{\to} \bmu$ as $t \to \infty$.
Then the continuity of $\bmu \mapsto \bw^* (\bmu)$, combined with the direct tracking step guarantees that almost surely the frequencies $\frac{N_a (t)}{t}$ converge to the optimal weights $w_a^* (\bmu)$ for all $a = 1, \ldots, K$.

\subsection{Stopping Rule: \texorpdfstring{$\pmb{(\alpha, \delta)}$}{Lg}-Chernoff's Stopping Rule}

For the stopping rule we will need the following statistics.
For any two distinct arms $a, b$ if $\hat{\mu}_a (N_a (t)) \ge \hat{\mu}_b (N_b (t))$, we define
\begin{align*}
Z_{a,b} (t) &=
\frac{N_a (t)}{N_a (t) + N_b (t)}
\KL{\hat{\mu}_a (N_a (t))}{\hat{\mu}_{a,b} (N_a (t), N_b (t))} + \\
& \qquad \frac{N_b (t)}{N_a (t) + N_b (t)}
\KL{\hat{\mu}_b (N_b (t))}{\hat{\mu}_{a,b} (N_a (t), N_b (t))},
\end{align*}
while if $\hat{\mu}_a (N_a (t)) < \hat{\mu}_b (N_b (t))$, we define
$Z_{a, b} (t) = - Z_{b, a} (t)$, where
\[
\hat{\mu}_{a,b} (N_a (t), N_b (t)) = 
\frac{N_a (t)}{N_a (t) + N_b (t)} \hat{\mu}_a (N_a (t)) +
\frac{N_b (t)}{N_a (t) + N_b (t)} \hat{\mu}_b (N_b (t)).
\]
Note that the statistics $Z_{a, b} (t)$ do not arise as the closed form solutions
of the Generalized Likelihood Ratio statistics for Markov chains, as it is the case in the IID bandits setting.

For a confidence level $\delta \in (0, 1)$, and a convergence parameter $\alpha > 1$ we define the $(\alpha, \delta)$-Chernoff stopping rule following~\cite{GK16}
\[
\tau_{\alpha, \delta} = \inf\: \left\{t \in \Intpp{} :
\exists a \in \{1,\ldots,K\} ~ \forall b \neq a, ~ Z_{a, b} (t) >
(0 \lor \beta_{\alpha, \delta} (t)) 
\right\},
\]
where $\beta_{\alpha, \delta} (t) = 2 \log \frac{D t^\alpha}{\delta}, ~ \dss D = \frac{2 \alpha K C^2}{\alpha-1}$,
and $C = C (P, f)$ is the constant from~\autoref{lem:evec-ratio}.
In the special case that $P$ is a positive stochastic matrix we can explicitly set $C = \max_{x,y,z}\: \frac{P (y, z)}{P (x, z)}$.
It is important to notice that the constant $C = C (P, f)$ does not depend on the bandit instance $\btheta$ or the confidence level $\delta$, but only on the generator stochastic matrix $P$ and the reward function $f$.
In other words it is a characteristic of the exponential family of Markov chains and not of the particular bandit instance, $\btheta$, under consideration.

\subsection{Decision Rule: Best Sample Mean}

For a fixed arm $a$ it is clear that, $\min_{b \neq a} Z_{a, b} (t) > 0$ if and only if
$\hat{\mu}_a (N_a (t)) > \hat{\mu}_b (N_b (t))$ for all $b \neq a$. Hence the following simple decision rule is well defined when used in conjunction with the $(\alpha, \delta)$-Chernoff stopping rule:
\[
\{\hat{a}_{\tau_{\alpha, \delta}}\} = \argmax_{a=1,\ldots,K}\: \hat{\mu}_a (N_a (\tau_{\tau_{\alpha, \delta}})).
\]

\subsection{Sample Complexity Analysis}

In this section we establish that the $(\alpha, \delta)$-Track-and-Stop strategy is $\delta$-PC, and we upper bound its expected sample complexity.
In order to do this we use our Markovian concentration bound~\autoref{thm:concentration-bound}.

We first use it in order to establish the following uniform deviation bound.
\begin{lemma}\label{lem:KL-concentation}
Let $\btheta \in \bTheta$, $\delta \in (0, 1)$, and $\alpha > 1$.
Let $\calA_\delta$ be a sampling strategy that uses
an arbitrary sampling rule, the $(\alpha, \delta)$-Chernoff's stopping rule and the best sample mean decision rule.
Then, for any arm $a$,
\[
\Pr_\btheta^{\calA_\delta} \left(\exists t \in \Intpp{} :
N_a (t) \KL{\hat{\mu}_a (N_a (t))}{\mu_a} \ge \beta_{\alpha, \delta} (t)/2\right) \le \frac{\delta}{K}.
\]
\end{lemma}
With this in our possession we are able to prove in~\autoref{app:upper-bound} that the $(\alpha, \delta)$-Track-and-Stop strategy is $\delta$-PC. 
\begin{proposition}\label{prop:PC}
Let $\delta \in (0, 1)$, and $\alpha \in (1, e/4]$.
The $(\alpha, \delta)$-Track-and-Stop strategy is $\delta$-PC.
\end{proposition}
Finally, we obtain that in the high confidence regime, $\delta \to 0$, the $(\alpha, \delta)$-Track-and-Stop strategy has a sample complexity which is at most $4 \alpha$ times the asymptotic lower bound that we established in~\autoref{thm:lower-bound}.
\begin{theorem}\label{thm:upper-bound}
Let $\btheta \in \bTheta$, and $\alpha \in (1, e/4]$.
The $(\alpha, \delta)$-Track-and-Stop strategy, denoted here by $\calA_{\delta}$, has its asymptotic expected sample complexity upper bounded by,
\[
\limsup_{\delta \to 0} \frac{\E_\btheta^{\calA_\delta} [\tau_{\alpha, \delta}]}{\log \frac{1}{\delta}} \le 4 \alpha T^* (\btheta).
\]
\end{theorem}

\section*{Acknowledgements} We would like to thank Venkat Anantharam, Jim Pitman and Satish Rao for many helpful discussions. This research was supported in part by the NSF grant CCF-1816861.

\bibliographystyle{apalike}
\bibliography{references}

\begin{appendices}
\section{Lower Bound on the Sample Complexity}\label{app:lower-bound}

We first prove~\autoref{lem:KL-bound},
for which we will apply a renewal argument.
Using the \emph{strong Markov property} we can derive the following standard, see~\cite{Durrett10}, decomposition of a Markov chain in IID blocks.
\begin{fact}\label{fact:iiid-blocks}
Let $\{X_n\}_{n \in \Intp{}}$ be an irreducible Markov chain
with initial distribution $q$, and transition matrix $P$.
Define recursively the $k$-th return time to the initial state as
\[
\begin{cases}
\tau_0 &= 0 \\
\tau_k &= \inf\: \{n > \tau_{k-1} : X_n = X_0\}, ~ \text{for} ~ k \ge 1,
\end{cases}
\]
and for $k \ge 1$ let $r_k = \tau_k - \tau_{k-1}$ be the residual time. Those random times partition the Markov chain in a sequence
$\{v_k\}_{k \in \Intpp{}}$ of
IID random blocks given by
\[
v_k = (r_k, X_{\tau_{k-1}}, \ldots, X_{\tau_k-1}),~\text{for}~ k \ge 1.
\]
\end{fact}
Let $N (x, n, m)$ be the number of visits to $x$ that occurred from time
$n$ up to time $m$, and $N (x, y, n, m)$ to be
the number of transitions from $x$ to $y$ that occurred from time
$n$ up to time $m$
\begin{align*}
N (x, n, m) &= \sum_{s=n}^{m-1} 1 \{X_s = x\}, \\
N (x, y, n, m) &= \sum_{s=n}^{m-1} 1 \{X_s = x, X_{s+1} = y\}.
\end{align*}
It is well know, see~\cite{Durrett10}, that the stationary distribution $\pi$ of the Markov chain is given by
\begin{equation}\label{eqn:stat-distr}
\pi (x) =
\frac{\E_{(q, P)} N (x, 0, \tau_1)}{\E_{(q, P)} \tau_1}, ~\text{for any}~ x \in S.
\end{equation}
In the following lemma we establish a similar relation for the invariant distribution over pairs of the Markov chain.
\begin{lemma}\label{lem:invariant}
\[
\pi (x) P (x, y) = 
\frac{\E_{(q, P)} N (x, y, 0, \tau_1)}{\E_{(q, P)} \tau_1},
~\text{for any}~ x, y \in S.
\]
\begin{proof}
Using~\eqref{eqn:stat-distr} it is enough to show that for any initial state $x_0$,
\[
\E_{(x_0, P)} N (x, 0, \tau_1) P (x, y) = \E_{(x_0, P)} N (x, y, 0, \tau_1),
\]
or equivalently that,
\[
\E_{(x_0, P)} \sum_{n=0}^{\tau_1-1} 1 \{X_n = x\} P (x, y) = 
\E_{(x_0, P)} \sum_{n=0}^{\tau_1-1} 1 \{X_n = x, X_{n+1} = y\}.
\]
Conditioning over the possible values of $\tau_1$, and using Fubini's Theorem we obtain
\begin{align*}
\E_{(x_0, P)} \sum_{n=0}^{\tau_1-1} 1 \{X_n = x\} P (x, y)
&= \sum_{t=1}^\infty \Pr_{x_0} (\tau_1 = t) \sum_{n=0}^{t-1} \Pr_{(x_0, P)} (X_n = x \mid \tau_1 = t) P (x, y) \\
&= \sum_{n=0}^\infty \sum_{t = n+1}^\infty \Pr_{(x_0, P)} (X_n = x, \tau_1 = t) P (x, y) \\
&= \sum_{n=0}^\infty \Pr_{(x_0, P)} (X_n = x, \tau_1 > n) P (x, y) \\
&= \sum_{n=0}^\infty \Pr_{(x_0, P)} (X_n = x, X_{n+1} = y) \Pr_{(x_0, P)} (\tau_1 > n \mid X_n = x) \\
&= \sum_{n=0}^\infty \Pr_{(x_0, P)} (X_n = x, X_{n+1} = y, \tau_1 > n) \\
&= \E_{(x_0, P)} \sum_{n=0}^{\tau_1-1} 1 \{X_n = x, X_{n+1} = y\},
\end{align*}
where the second to last equality holds true due to the reversed Markov property
\[
\Pr_{(x_0, P)} (\tau_1 > n \mid X_n = x, X_{n+1} = y) = 
\Pr_{(x_0, P)} (\tau_1 > n \mid X_n = x).
\]
\end{proof}
\end{lemma}
The following Lemma, which is a variant of Lemma 2.1 in~\cite{Ananth-Varaiya-Walrand-II-87}, is the place where we use the IID block structure of the Markov chain.
\begin{lemma}\label{lem:renewal}
Define the mean return time of the Markov chain with initial distribution $q$ and irreducible transition matrix $P$ by
\[
R = \E_{(q, P)} [\inf\: \{n > 0 : X_n = X_0\}] < \infty.
\]
Let $\calF_n$ be the $\sigma$-field generated by $X_0, X_1, \ldots, X_n$.
Let $\tau$ be a stopping time with respect to $(\calF_n)_{n \in \Intp{}}$, with $\E_{(q, P)} \tau < \infty$. 
Then
\[
\E_{(q, P)} N (x, y, 0, \tau) \le \pi (x) P (x, y) (\E_{(q, P)} \tau + R - 1), ~ \text{for all} ~ x, y \in S.
\]
\begin{proof}
Using the $k$-th return times from~\autoref{fact:iiid-blocks}
we decompose $N (x, y, 0, \tau_k)$ in $k$ IID summands
\[
N (x, y, 0, \tau_k) =
\sum_{i=0}^{k-1} N (x, y, \tau_i, \tau_{i+1}).
\]
Now let $\kappa = \inf\: \{k > 0 : \tau_k \ge \tau\}$, so that $\tau_\kappa$ is the first return time to the initial state after or at time $\tau$.
By definition of $\tau_\kappa$ we have that
\[
\tau_\kappa - \tau \le \tau_\kappa - \tau_{\kappa-1} - 1.
\]
Taking expectations we obtain
\[
\E_{(q, P)} [\tau_\kappa - \tau] \le
\E_{(q, P)} [\tau_\kappa - \tau_{\kappa-1}] - 1= 
\E_{(q, P)} r_\kappa - 1 =
\E_{(q, P)} r_1 - 1 = R - 1,
\]
which also gives that
\[
\E_{(q, P)} [\tau_\kappa] \le \E_{(q, P)} [\tau] + R - 1 < \infty.
\]
This allows us to use Wald's identity, followed by~\autoref{lem:invariant},
followed by Wald's identity again, in order to get
\begin{align*}
\E_{(q, P)} N (x, y, 0, \tau_\kappa)
&= \E_{(q, P)} \sum_{i=0}^{\kappa-1} N (x, y, \tau_i, \tau_{i+1}) \\
&= \E_{(q, P)} [N (x, y, 0, \tau_1)] \E_q [\kappa] \\
&= p (x) P (x, y) \E_{(q, P)} [\tau_1] \E_{(q, P)} [\kappa] \\
&= p (x) P (x, y) \E_{(q, P)} [\tau_\kappa].
\end{align*}
Therefore,
\begin{align*}
\E_{(q, P)} N (x, y, 0, \tau) 
&\le \E_{(q, P)} N (x, y, 0, \tau_\kappa) \\
&= \pi (x) P (x, y) \E_{(q, P)} [\tau_\kappa] \\
&\le \pi (x) P (x, y) (\E_{(q, P)} [\tau] + R - 1).
\end{align*}
\end{proof}
\end{lemma}
\begin{proof}[Proof of~\autoref{lem:KL-bound}]\hfill\break
Follows by taking $\Etheta$ of the log-likelihood ratio,
$\log \left(\frac{\Prtheta \mid_{\calF_\tau}}
{\Prlambda \mid_{\calF_\tau}}\right)$, given by~\eqref{eqn:log-ratio}, and applying~\autoref{lem:renewal} $K$ times for the stopping times $N_a (\tau) + 1,~ a = 1, \ldots, K$.
\end{proof}
The last part of~\autoref{app:lower-bound} involves the proof of~\autoref{thm:lower-bound}.
\begin{proof}[Proof of~\autoref{thm:lower-bound}]\hfill\break
Consider an alternative parametrization $\blambda \in \Atheta$.
The data processing inequality, see~\cite{CovThom06},
gives us as a way to lower bound the Kullback-Leibler divergence between the two probability measures $\Prtheta \mid_{\calF_{\tau_\delta}}$ and $\Prlambda \mid_{\calF_{\tau_\delta}}$.
In particular,
\[
\KLb{\Prtheta (\calE)}{\Prlambda (\calE)} \le 
\KL{\Prtheta \mid_{\calF_{\tau_\delta}}}{\Prlambda \mid_{\calF_{\tau_\delta}}},
~\text{for any}~ \calE \in \calF_{\tau_\delta},
\]
where for $p, q \in [0, 1]$, $\KLb{p}{q}$ denotes the binary Kullback-Leibler divergence,
\[
\KLb{p}{q} = p \log \frac{p}{q} + (1-p) \log \frac{1-p}{1-q}.
\]
We apply this inequality with the event $\calE = \{\hat{a}_{\tau_\delta} \neq a^* (\btheta)\} \in \calF_{\tau_\delta}$. The fact that the strategy $\calA_\delta$ is $\delta$-PC implies that
\[
\Pr_\btheta (\calE) \le \delta, \quad\text{and}\quad \Pr_\blambda (\calE) \ge 1 - \delta,
\]
hence
\[
\KLb{\delta}{1-\delta}
\le \KL{\Prtheta \mid_{\calF_{\tau_\delta}}}{\Prlambda \mid_{\calF_{\tau_\delta}}}.
\]
Combining this with~\autoref{lem:KL-bound} we get that
\[
\KLb{\delta}{1-\delta} \le
\sum_{a=1}^K \KL{q_{\theta_a}}{q_{\lambda_a}} +
\sum_{a=1}^K \left(\Etheta [N_a (\tau_\delta)] + R_a\right) \KL{\theta_a}{\lambda_a}.
\]
The fact that $\sum_{a=1}^K N_a (\tau_\delta) \le \tau_\delta$ gives,
\begin{align*}
& \KLb{\delta}{1-\delta} - 
\sum_{a=1}^K \KL{q_{\theta_a}}{q_{\lambda_a}} \\
&\qquad\le \left(\Etheta [\tau_\delta] + \sum_{a=1}^K R_a\right)
\sum_{a=1}^K
\frac{\Etheta [N_a (\tau_\delta)] + R_a}
{\sum_{b=1}^K \left(\Etheta [N_b (\tau_\delta)] + R_b\right)} \KL{\theta_a}{\lambda_a},
\end{align*}
and now we follow the technique of~\cite{GK16} which combines multiple alternative models $\blambda$,
\begin{align*}
& \KLb{\delta}{1-\delta} - 
\sum_{a=1}^K \KL{q_{\theta_a}}{q_{\lambda_a}} \\
&\qquad\le 
\left(\Etheta [\tau_\delta] + \sum_{a=1}^K R_a\right)
\inf_{\blambda \in \Atheta}\sum_{a=1}^K
\frac{\Etheta [N_a (\tau_\delta)] + R_a}
{\sum_{b=1}^K \left(\Etheta [N_b (\tau_\delta)] + R_b\right)} \KL{\theta_a}{\lambda_a}
\\
&\qquad\le \left(\Etheta [\tau_\delta] + \sum_{a=1}^K R_a\right)
\sup_{\bw \in \Simplex{[K]}} \inf_{\blambda \in \Atheta}\: \sum_{a=1}^K
w_a \KL{\theta_a}{\lambda_a}.
\end{align*}
The conclusion follows by letting $\delta$ go to $0$, and using the fact that
\[
\lim_{\delta \to 0} \frac{\KLb{\delta}{1-\delta}}{\log \frac{1}{\delta}} = 1.
\]
\end{proof}
\section{Exponential Family of Stochastic Matrices}\label{app:exp-fam}

For a stochastic matrix $P$ on $S$, and a probability distribution $p \in \Simplex{S}$, we use the notation $p \odot P \in \Simplex{S \times S}$ to denote the bivariate distribution on $S \times S$ given by
\[
(p \odot P) (x, y) = p (x) P (x, y).
\]
We start by establishing parts (a), (b) and (c) of~\autoref{lem:dual-map}.
\begin{proof}[Proof of~\autoref{lem:dual-map}]\hfill
\begin{enumerate}[label=(\alph*)]
\item Each entry of $\tilde{P}_\theta$ is a real analytic function
    of $\theta$, and for each $\theta_0$ the Perron-Frobenius eigenvalue
    $\rho (\theta_0)$ is simple with a unique corresponding 
    left and right eigenvectors $u_{\theta_0}, ~v_{\theta_0}$ and such that they are both positive,
    $\sum_x u_{\theta_0} (x) = 1$ and $\sum_x u_{\theta_0} (x) v_{\theta_0} (x) = 1$.
    The conclusion follows by standard implicit function theorem type of arguments. See for example Theorem 7 and Theorem 8 in Chapter 9 from the book of~\cite{Lax-07}.
\item For any $x, y \in S$ such that $P (x, y) > 0$ we have that
    \[
    \log P_\theta (x, y) = \theta f (y) - A (\theta) + \log v_\theta (y) - \log v_\theta (x) + \log P (x, y).
    \]
    Differentiating with respect to $\theta$, and taking expectation with respect to $\pi_\theta \odot P_\theta$ we obtain
    \[
    \E_{(X, Y) \sim \pi_\theta \odot P_\theta} \frac{d}{d \theta} \log P_\theta (X, Y)
    = \pi_\theta (f) - \dot{A} (\theta),
    \]
    where the logarithms cancel out since $\pi_\theta \odot P_\theta$ has identical marginals. The conclusion follows because
    \[
    \E_{(X, Y) \sim \pi_\theta \odot P_\theta} \frac{d}{d \theta} \log P_\theta (X, Y)
    = \sum_x \pi_\theta (x) \frac{d}{d \theta} \left(\sum_y P_\theta (x, y)\right)
    = 0.
    \]
\item For any $x, y \in S$ such that $P (x, y) > 0$ we have that
    \[
    \frac{d^2}{d \theta^2} \log P_\theta (x, y) =
    - \ddot{A} (\theta) + \frac{d^2}{d \theta^2} \log v_\theta (y) -
    \frac{d^2}{d \theta^2} \log v_\theta (x).
    \]
    Taking expectation with respect to $\pi_\theta \odot P_\theta$ we obtain
    \begin{align*}
    \ddot{A} (\theta)
    &= - \E_{(X, Y) \sim \pi_\theta \odot P_\theta}\frac{d^2}{d \theta} \log P_\theta (X, Y) \\
    &= \E_{(X, Y) \sim \pi_\theta \odot P_\theta}\left(\frac{d}{d \theta} \log P_\theta (X, Y)\right)^2 \ge 0.
    \end{align*}
    This ensures that $\dot{A} (\theta)$ is increasing.
    
    Assume, towards contradiction, that $\ddot{A} (\theta) = 0$ in a neighborhood of  $\theta_0$.
    Then $P_\theta$ does not depend on $\theta$ in a neighborhood of  $\theta_0$.
    The $S_M$ component is irreducible so we can find 
    $x_1, \ldots, x_{l+1} \in S_M$ such that $P (x_i, x_{i+1}) > 0$  for $i=1, \ldots, l$ and $x_1 = x_{l+1}$, and so
    \[
    P_\theta (x_1, x_2) \ldots P_\theta (x_l, x_{l+1})
    = \frac{P (x_1, x_2) \ldots P (x_l, x_{l+1}) e^{\theta l M}}{\rho (\theta)^l},
    \]
    and the $S_m$ component is irreducible as well so we can find 
    $y_1, \ldots, y_{k+1} \in S_m$ such that $P (y_i, y_{i+1}) > 0$  for $i=1, \ldots, k$ and $y_1 = y_{k+1}$, and so
    \[
    P_\theta (y_1, y_2) \ldots P_\theta (y_l, y_{k+1})
    = \frac{P (y_1, y_2) \ldots P (y_k, y_{k+1}) e^{\theta k m}}{\rho (\theta)^k}.
    \]
    This means that the ratio
    \[
    \frac
    {(P_\theta (x_1, x_2) \cdots P_\theta (x_l, x_{l+1}))^{1/l}}
    {(P_\theta (y_1, y_2) \cdots P_\theta (y_k, y_{k+1}))^{1/k}} = \frac{P (x_1, x_2) \cdots P (x_l, x_{l+1})}{P (y_1, y_2) \cdots P (y_k, y_{k+1})} e^{\theta (M-m)},
    \]
    depends on $\theta$.
    This contradicts the assumption that $P_\theta$ does not depend on $\theta$ on a neighborhood of $\theta_0$.
    
    Therefore, $\ddot{A} (\theta)$ does not vanish on any nonempty open interval
    of $\Real{}$, and so we conclude that $\dot{A} (\theta)$ is strictly increasing.
\end{enumerate}
\end{proof}
Showing part (d) of~\autoref{lem:dual-map} requires the study of the limiting behavior of the family which we do in the following two Lemmata. The first is a simple extension of the Perron-Frobenius theory.
\begin{lemma}\label{lem:PF}
Let $W \in \Realp{n \times n}$ be a non-negative matrix
consisting of:
a non-negative irreducible square block $A \in \Realp{k \times k}$,
and a non-negative rectangular block $B \in \Realp{(n-k) \times k}$
such that none of the rows of $B$ is zero,
for some $k \in \{1, \ldots, n\}$,
assembled together in the following way:
\[
W =
\begin{bmatrix}
A & 0 \\
B & 0
\end{bmatrix},
\]
Then, $\rho (W) = \rho (A)$ is a simple eigenvalue of $W$,
which we call the Perron-Frobenius eigenvalue, and is associated with unique left and right eigenvectors $u_W, v_W$ such that $u_W$ has its first $k$ coordinates positive
and its last $n-k$ coordinates equal to zero,
$v_W$ is positive, $\sum_{x=1}^n u_W (x) = 1$, and
$\sum_{x=1}^n u_W (x) v_W (x) = 1$.
\begin{proof}
Let $u_A, v_A$ be the unique left and right eigenvectors of $A$ corresponding to the Perron-Frobenius eigenvalue $\rho (A)$,
such that both of them are positive, $\sum_{x=1}^k u_A (x) = 1$ and $\sum_{x=1}^k u_A (x) v_A (x) = 1$. Observe that the vectors
\[
u_W =
\begin{bmatrix}
u_A \\
0
\end{bmatrix}
,~\text{and}~
v_W =
\begin{bmatrix}
v_A \\
B v_A / \rho (A)
\end{bmatrix},
\]
are left and right eigenvectors of $W$ with associated eigenvalue $\rho (A)$, and satisfy all the conditions.
In addition, $\rho (W)$ being greater than $\rho (A)$, or $\rho (W)$ not being a simple eigenvalue, or $u_W, v_W$ not being unique would contradict the Perron-Frobenius Theorem for the nonnegative irreducible matrix $A$.
\end{proof}
\end{lemma}
Now we define the matrix $\overline{P}_\infty = \lim_{\theta \to \infty} e^{-\theta M} \tilde{P}_\theta$, i.e. the matrix $P$ where we keep the columns $y \in S_M$ intact, and we zero out all the other columns.
After suitable permutation of the states~\autoref{lem:PF} applies for $\overline{P}_\infty$, and so
$\rho (\overline{P}_\infty)$ is a simple eigenvalue of 
$\overline{P}_\infty$, which is associated with unique left and right eigenvectors $u_\infty, v_\infty$ such that $u_\infty (x) > 0$ for $x \in S_M$ and $u_\infty (x) = 0$ for $x \not\in S_M$, $v_\infty$ is positive, $\sum_x u_\infty (x) = 1$ and $\sum_x u_\infty (x) v_\infty (x) = 1$.
Similarly, we define $\overline{P}_{-\infty} := \lim_{\theta \to -\infty} e^{-\theta m} \tilde{P}_\theta$,
with Perron-Frobenius eigenvalue $\rho (\overline{P}_{-\infty})$,
which is associated with unique left and right eigenvectors $u_{-\infty}, v_{-\infty}$ such that
$u_{-\infty} (x) > 0$ for $x \in S_m$
and $u_{-\infty} (x) = 0$ for $x \not\in S_m$, $v_{-\infty}$ is positive,
$\sum_x u_{-\infty} (x) = 1$ and $\sum_x u_{-\infty} (x) v_{-\infty} (x) = 1$.

The following Lemma characterizes the limiting stochastic matrices $P_\infty, ~P_{-\infty}$ of the exponential family,
and proves part (d) of~\autoref{lem:dual-map}.
\begin{lemma}\label{lem:limit}\hfill
\begin{enumerate}[label=(\alph*)]
\item $\theta M - A (\theta) \to
- \log \rho (\overline{P}_\infty), ~ u_\theta \to u_\infty, ~ v_\theta \to v_\infty$, as $\theta \to \infty$, and so
\[
\lim_{\theta \to \infty} P_\theta (x, y) =
\frac{\overline{P}_\infty (x, y) v_\infty (y)}{\rho (\overline{P}_\infty) v_\infty (x)} =: P_\infty (x, y),
\]
and $\pi_\theta (f) \to M$ as $\theta \to \infty$.

\item
$\theta m - A (\theta) \to
- \log \rho (\overline{P}_{-\infty}), ~ u_\theta \to u_{-\infty}, ~ v_\theta \to v_{-\infty}$, as $\theta \to -\infty$, and so
\[
\lim_{\theta \to - \infty} P_\theta (x, y) = \frac{\overline{P}_{-\infty} (x, y) v_{-\infty} (y)}{\rho (\overline{P}_{-\infty})v_{-\infty} (x)}
=: P_{-\infty} (x, y),
\]
and $\pi_\theta (f) \to m$ as $\theta \to - \infty$.
\end{enumerate}
\begin{proof} Both parts are a straightforward application 
of the continuity of the function
$P \mapsto (\rho (P), ~u_P, ~v_P)$,
at $\overline{P}_\infty$ and $\overline{P}_{-\infty}$.
The continuity of eigenvalues and eigenvectors is due to the fact that the Perron-Frobenius eigenvalue $\rho (P)$ is a simple eigenvalue and more details can be found in Chapter 3 of the book
~\cite{Ortega90}.
\end{proof}
\end{lemma}
This lemma suggests that we can extend the domain of $\dot{A} (\theta)$ by continuity over the set of extended real numbers $\overline{\Real{}} = \Real{} \cup \{\pm \infty\}$, by defining $\dot{A} (\infty) = M$ and
$\dot{A} (-\infty) = m$. This way we have a one-to-one and onto correspondence of  $\overline{\Real{}}$ with the closed interval $[m, M]$, with the limit stochastic matrices being $P_\infty$
and $P_{-\infty}$, which represent degenerate Markov chains where all the transitions lead into states $y \in S_M$ when $\theta = \infty$,
and into states $y \in S_m$ when $\theta = -\infty$.

We proceed by deriving some alternative representations
for the Kullback-Leibler divergence rate between elements of the exponential family. The following lemma is needed in order to derive the asymptotic Kullback-Leibler divergence rate.
\begin{lemma}\label{lem:limit-KL}\hfill
\begin{enumerate}[label=(\alph*)]
\item $\theta \dot{A} (\theta) - A (\theta) \to
- \log \rho (\overline{P}_\infty)$, as $\theta \to \infty$.
    
\item $\theta \dot{A} (\theta) - A (\theta) \to
- \log \rho (\overline{P}_{-\infty})$, as $\theta \to -\infty$.
    
\end{enumerate}
\begin{proof}
Let $M_2 = \max_{x \not \in S_M} f (x)$. 
Fix $x \in S$ and $y \not\in S_M$.
Pick $y_M \in S_M$ such that $P (x, y_M) > 0$.

Using~\autoref{lem:evec-ratio} we see that there is a constant
$C = C (P, f)$ such that
\[
P_\theta (x, y)
\le C e^{- \theta (M-f (y))} P_\theta (x, y_M)
\le C e^{- \theta (M-M_2)}.
\]
Therefore the stationary probability of any such $y$ is at most
$\pi_\theta (y) \le C e^{- \theta (M-M_2)}$, and so
\[
\pi_\theta (f) \ge (1- C |S| e^{- \theta (M-M_2)}) M + C |S| e^{- \theta (M-M_2)} m.
\]
From this we obtain that
\[
0 \le \theta (M - \pi_\theta (f)) \le C |S| \theta e^{- \theta (M-M_2)} (M - m), ~ \text{for any} ~ \theta \ge 0,
\]
which yields that $\theta (\dot{A} (\theta) - M) \to 0$, as $\theta \to \infty$. Part (a) now follows, since~\autoref{lem:limit} suggests that
$\theta M - A (\theta) \to - \log \rho (\overline{P}_\infty)$, as $\theta \to \infty$.
The second limit follows by the same argument.
\end{proof}
\end{lemma}
Having this in our possession we state and prove alternative representations for the Kullback-Leibler divergence rate.
\begin{lemma}\label{lem:KL}\hfill
\begin{enumerate}[label=(\alph*)]
    \item For all $\theta_1, \theta_2 \in \Real{}$,
    \begin{align*}
    \KL{\theta_1}{\theta_2} &= \theta_1 \dot{A} (\theta_1) - A (\theta_1) -
    (\theta_2 \dot{A} (\theta_1) - A (\theta_2)); \\
    \KL{\infty}{\theta_2}&= - \log \rho (\overline{P}_\infty) - (\theta_2 M - A (\theta_2)); \\
    \KL{-\infty}{\theta_2}&= - \log \rho (\overline{P}_{-\infty}) - (\theta_2 m - A (\theta_2)).
    \end{align*}
    
    \item For all $\mu_1, \mu_2 \in (m, M)$,
    \begin{align*}
    \KL{\mu_1}{\mu_2}
    &= \dot{A}\inv (\mu_1) \mu_1 - A (\dot{A}\inv (\mu_1)) - (\dot{A}\inv (\mu_2) \mu_1 - A (\dot{A}\inv (\mu_2))); \\
    \KL{M}{\mu_2}
    &= -\log \rho (\overline{P}_\infty) - (\dot{A}\inv (\mu_2)M-A (\dot{A}\inv (\mu_2))); \\
    \KL{m}{\mu_2} &= 
    -\log \rho (\overline{P}_{-\infty}) - (\dot{A}\inv (\mu_2)m-A (\dot{A}\inv (\mu_2))).
    \end{align*}
\end{enumerate}
\begin{proof}
For $\theta_1, \theta_2 \in \Real{}$ we have that
\begin{align*}
\KL{\theta_1}{\theta_2}
&= \E_{(X,Y) \sim \pi_{\theta_1} \odot P_{\theta_1}}
\log \frac{P_{\theta_1} (X, Y)}{P_{\theta_2} (X, Y)} \\
&= A (\theta_2) - A (\theta_1) - (\theta_2 - \theta_1) \dot{A} (\theta_1)
+ \E_{(X,Y) \sim \pi_{\theta_1} \odot P_{\theta_1}} 
\left[\log \frac{v_{\theta_1} (Y)}{v_{\theta_1} (X)} -
\log \frac{v_{\theta_2} (Y)}{v_{\theta_2} (X)}\right] \\
&= \theta_1 \dot{A} (\theta_1) - A (\theta_1) -
    (\theta_2 \dot{A} (\theta_1) - A (\theta_2)),
\end{align*}
and the third equality follows due to the fact that $\pi_{\theta_1} \odot P_{\theta_1}$ has identical marginals and so the expectation vanishes.

Now let $\theta_2 \in \Real{}$.
Using the continuity of the Kullback-Leibler divergence rate, 
the formula that we just established,
and~\autoref{lem:limit-KL} we obtain
\begin{align*}
\KL{\infty}{\theta}
&= \lim_{\theta_1 \to \infty} \KL{\theta_1}{\theta_2} \\
&= \lim_{\theta_1 \to \infty}
\left(\theta_1 \dot{A} (\theta_1) - A (\theta_1)\right)
- \lim_{\theta_1 \to \infty} 
\left(\theta_2 \dot{A} (\theta_1) - A (\theta_2)\right) \\
&= - \log \rho (\overline{P}_\infty) - (\theta_2 M - A (\theta_2)).
\end{align*}
We argue in the same way for $\KL{-\infty}{\theta}$, and
part (b) directly follows from part (a).
\end{proof}
\end{lemma}
As a direct consequence of these representation we obtain the following monotonicity properties of the Kullback-Leibler divergence rate.
\begin{corollary}\label{cor:incr}\hfill
\begin{enumerate}[label=(\alph*)]
    \item For fixed $\theta_2 \in \Real{}$, the function $\theta_1 \mapsto \KL{\theta_1}{\theta_2}$ is strictly increasing in the interval $[\theta_2, \infty]$ and strictly decreasing in the interval $[-\infty, \theta_2]$.
    
    \item For fixed $\mu_2 \in (m, M)$, the function $\mu_1 \mapsto \KL{\mu_1}{\mu_2}$ is strictly increasing in the interval $[\mu_2, M]$ and strictly decreasing in the interval $[m, \mu_2]$.
\end{enumerate}
\end{corollary}
We close this appendix by establishing that the Kullback-Leibler divergence rate
is the convex conjugate of the log-Perron-Frobenius eigenvalue.
\begin{lemma}\label{lem:conv-conj}
\[
\KL{\mu}{\mu (0)} =
\sup_{\theta \in \Real{}}\:\{\theta \mu - A (\theta)\} =
\begin{cases}
\dss \sup_{\theta \ge 0}\:\{\theta \mu - A (\theta)\}, &
\text{if} ~ \mu \in [\mu (0), M] \\
\dss \sup_{\theta \le 0}\:\{\theta \mu - A (\theta)\}, &
\text{if} ~ \mu \in [m, \mu (0)].
\end{cases}
\]
\begin{proof}
Fix $\mu \in (m, M)$. The function
$\theta \mapsto \theta \mu - A (\theta)$ is strictly concave
and its derivative vanishes at $\theta = \dot{A}\inv (\mu)$,
which belong in $[0, \infty)$ when $\mu \in [\mu (0), M)$
and in $(-\infty, 0]$ when $\mu \in (m, \mu (0)]$.
Therefore, using~\autoref{lem:KL} we obtain
\[
\sup_{\theta \in \Real{}}\: \{\theta \mu - A (\theta)\}
= \dot{A}\inv (\mu) \mu - A (\dot{A}\inv (\mu)) = \KL{\mu}{\pi (f)}.
\]
Similarly when $\mu = M$ or $\mu = m$, the derivative only vanishes at $\infty$ and $-\infty$ respectively, and so from a combination of~\autoref{lem:limit} and~\autoref{lem:KL} we obtain
\[
\sup_{\theta \in \Real{}}\: \{\theta M - A (\theta)\}
= \lim_{\theta \to \infty} (\theta M - A (\theta)) = \KL{M}{\pi (f)},
\]
and
\[
\sup_{\theta \in \Real{}}\: \{\theta m - A (\theta)\}
= \lim_{\theta \to -\infty} (\theta m - A (\theta)) = \KL{m}{\pi (f)}.
\]
\end{proof}
\end{lemma}
\section{Concentration for Markov Chains}
\label{app:concentration}

We first use continuity in order to get a uniform bound on the ratio of the entries of the right Perron-Frobenius eigenvector.
\begin{lemma}\label{lem:evec-ratio}
Let $P$ be an irreducible stochastic matrix on $S$,
which combined with $f : S \to \Real{}$ satisfies~\eqref{eqn:as-up-1},~\eqref{eqn:as-up-2},~\eqref{eqn:as-low-1}, and~\eqref{eqn:as-low-2}. 
There exists a constant $C = C (P, \phi) \ge 1$ such that
\[
C\inv \le \sup_{\theta \in \Real{}, x, y \in S}\: \frac{v_\theta (y)}{v_\theta (x)} \le C.
\]
If in addition $P$ is a positive stochastic matrix then we can take
$C = \max_{x, y, z} \frac{P (y, z)}{P (x, z)}$.
\begin{proof}
For any $x, y \in S$, the ratio $\frac{v_\theta (y)}{v_\theta (x)}$ is a positive real number, and due to~\autoref{lem:dual-map} a continuous function of $\theta$. In addition~\autoref{lem:PF} and~\autoref{lem:limit} suggest that its limit points $\frac{v_\infty (y)}{v_\infty (x)}, ~ \frac{v_{-\infty} (y)}{v_{-\infty} (x)}$ are positive real numbers as well,
hence we can take $C = \sup_{\theta \in \Real{}, x, y \in S}\: \frac{v_\theta (y)}{v_\theta (x)} \ge 1$, which is guaranteed to be finite.

In the special case that $P$ is a positive stochastic matrix, we use the fact that $v_\theta$ is a right Perron-Frobenius eigenvector of $\tilde{P}_\theta$ in order to write
\[
\frac{v_\theta (y)}{v_\theta (x)} =
\frac{\sum_w \tilde{P}_\theta (y, w) v_\theta (w)}{\sum_w \tilde{P}_\theta (x, w) v_\theta (w)}, ~ \text{for all} ~ x, y \in S.
\]
Now using the simple inequality
\[
\left(\min_z \frac{\tilde{P}_\theta (y, z)}{\tilde{P}_\theta (x, z)}\right) \tilde{P}_\theta (x, w)
\le \tilde{P}_\theta (y, w)
\le \left(\max_z \frac{\tilde{P}_\theta (y, z)}{\tilde{P}_\theta (x, z)}\right)
\tilde{P}_\theta (x, w), ~ \text{for all} ~ x,y,w \in S,
\]
and observing that $\frac{\tilde{P}_\theta (y, z)}{\tilde{P}_\theta (x, z)} = \frac{P (y, z)}{P (x, z)}$ we obtain
\[
\min_z \frac{P (y, z)}{P (x, z)}
\le \frac{v_\theta (y)}{v_\theta (x)}
\le \max_z \frac{P (y, z)}{P (x, z)}.
\]
\end{proof}
\end{lemma}
Next we establish a Proposition which gives us an approximation of the log-Perron-Frobenius eigenvalue using the
log-moment-generating-function 
\[
A_n (\theta) =
\frac{1}{n} \log \E_0 \exp \left\{\theta (\phi (X_1) + \ldots + \phi (X_n))\right\}
\]
\begin{proposition}\label{prop:log-mgf-estim}
Let $P$ be an irreducible stochastic matrix on $S$,
which combined with $f : S \to \Real{}$ satisfies~\eqref{eqn:as-up-1},~\eqref{eqn:as-up-2},~\eqref{eqn:as-low-1}, and~\eqref{eqn:as-low-2}. Then
\[
|A_n (\theta) - A (\theta)| \le \frac{\log C}{n}, ~ \text{for all} ~ \theta \in \Real{},
\]
where $C = C (P, f)$ is the constant from~\autoref{lem:evec-ratio}.
\begin{proof}
We start with the following calculation
\begin{align*}
e^{n A_n (\theta)} 
&= \sum_{x_0, x_1, \ldots, x_{n-1}, x_n} q (x_0) P (x_0, x_1) e^{\theta \phi (x_1)} \cdots 
P (x_{n-1}, x_n) e^{\theta \phi (x_n)} \\
&= \sum_{x_0, x_n} q (x_0) \tilde{P}_\theta^n (x_0, x_n).
\end{align*}
From this using the simple inequality
\[
\frac{v_\theta (y)}{\max_x v_\theta (x)} \le 1 \le \frac{v_\theta (y)}{\min_x v_\theta (x)}, ~ \text{for all} ~ y \in S,
\]
together with the fact that $v_\theta$ is a right Perron-Frobenius eigenvector of $\tilde{P}_\theta$ we obtain
\[
\min_{x, y} \frac{v_\theta (y)}{v_\theta (x)} e^{n A (\theta)}
\le e^{n A_n (\theta)}
\le \max_{x, y} \frac{v_\theta (y)}{v_\theta (x)} e^{n A (\theta)}.
\]
The conclusion now follows by applying~\autoref{lem:evec-ratio}
\end{proof}
\end{proposition}

One more ingredient that we need is a uniform bound of the constant $C (P_\theta, f)$ over $\theta \in \Real{}$.
\begin{lemma}\label{lem:uniform}
For the constant from~\autoref{lem:evec-ratio} we have that,
\[
\sup_{\theta \in \Real{}}\: C (P_\theta, f) \le C (P, f)^2.
\]
\begin{proof}
Recall that
\[
C (P_{\theta_2}, f) =
\sup_{\theta_1 \in \Real{}, x, y \in S}\: \frac{v_{\widetilde{\left(P_{\theta_2}\right)}_{\theta_1}} (y)}
{v_{\widetilde{\left(P_{\theta_2}\right)}_{\theta_1}} (x)}.
\]
We claim that
\[
\frac{v_{\widetilde{\left(P_{\theta_2}\right)}_{\theta_1}} (y)}
{v_{\widetilde{\left(P_{\theta_2}\right)}_{\theta_1}} (x)}
=
\frac{v_{\tilde{P}_{\theta_1 + \theta_2}} (y)
v_{\tilde{P}_{\theta_2}} (x)}
{v_{\tilde{P}_{\theta_1 + \theta_2}} (x)
v_{\tilde{P}_{\theta_2}} (y)}.
\]
To see this we just need to verify that
\[
v_{\tilde{P}_{\theta_2}} (x)
v_{\widetilde{\left(P_{\theta_2}\right)}_{\theta_1}} (x),
~ x \in S,
\]
is a right eigenvector of $\tilde{P}_{\theta_1 + \theta_2}$,
with associated eigenvalue $\rho (\tilde{P}_{\theta_2})
\rho \left(\widetilde{\left(P_{\theta_2}\right)}_{\theta_1}\right)$, which from the Perron-Frobenious theory has to be the Perron-Frobenious eigenvalue since the associated eigenvector has positive entries. The verification is straight forward
\begin{align*}
\sum_y \tilde{P}_{\theta_1 + \theta_2} (x, y)
v_{\tilde{P}_{\theta_2}} (y)
v_{\widetilde{\left(P_{\theta_2}\right)}_{\theta_1}} (y)
&= \rho (\tilde{P}_{\theta_2}) v_{\tilde{P}_{\theta_2}} (x)
\sum_y \widetilde{\left(P_{\theta_2}\right)}_{\theta_1} (x, y)
v_{\widetilde{\left(P_{\theta_2}\right)}_{\theta_1}} (y) \\
&= \rho (\tilde{P}_{\theta_2})
\rho \left(\widetilde{\left(P_{\theta_2}\right)}_{\theta_1}\right) v_{\tilde{P}_{\theta_2}} (x)
v_{\widetilde{\left(P_{\theta_2}\right)}_{\theta_1}} (x),
~\text{for all}~ x \in S.
\end{align*}

From this we see that
\[
\sup_{\theta_1, \theta_2 \in \Real{}, x, y \in S} 
\frac{v_{\widetilde{\left(P_{\theta_2}\right)}_{\theta_1}} (y)}
{v_{\widetilde{\left(P_{\theta_2}\right)}_{\theta_1}} (x)} \le 
\left(
\sup_{\theta_1, \theta_2 \in \Real{}, x, y \in S}  
\frac{v_{\tilde{P}_{\theta_1 + \theta_2}} (y)}
{v_{\tilde{P}_{\theta_1 + \theta_2}} (x)}
\right)
\left(
\sup_{\theta_2 \in \Real{}, x, y \in S} 
\frac{v_{\tilde{P}_{\theta_2}} (x)}
{v_{\tilde{P}_{\theta_2}} (y)}
\right) =
C (P, f)^2.
\]
\end{proof}
\end{lemma}

We are now ready to prove~\autoref{thm:concentration-bound}.
\begin{proof}[Proof of~\autoref{thm:concentration-bound}]\hfill\break
We first prove the bound for $\theta = 0$.
Fix $\mu \in [\mu (0), M]$, and $\eta \ge 0$.
\begin{align*}
\Pr_0 \left(f (X_1) + \ldots + f (X_n) \ge n \mu\right)
&\le \Pr_0 \left(e^{\eta(f (X_1) + \ldots + f (X_n))} \ge e^{\eta n \mu}\right) \\
&\le e^{- n (\eta \mu - A_n (\eta))} \\
&\le C (P, f) e^{- n (\eta \mu - A (\eta))},
\end{align*}
where the second inequality is Markov's inequality, and the third is the estimate from~\autoref{prop:log-mgf-estim}.
By optimizing over $\eta \ge 0$ and applying~\autoref{lem:conv-conj}, we obtain
\[
\Pr_0 \left(f (X_1) + \ldots + f (X_n) \ge n \mu\right) \le C (P, f)
e^{- n \KL{\mu}{\mu (0)}}.
\]
Applying this bound with $P_\theta$ in place of $P$, and using~\autoref{lem:uniform} we conclude that for $\mu \in [\mu (\theta), M]$
\[
\Pr_\theta \left(f (X_1) + \ldots + f (X_n) \ge n \mu\right) \le C (P_\theta, f)
e^{- n \KL{\mu}{\mu (\theta)}} \le C (P, f)^2
e^{- n \KL{\mu}{\mu (\theta)}}.
\]
\end{proof}

\section{Upper Bound on the Sample Complexity: the \texorpdfstring{$\pmb{(\alpha, \delta)}$}{Lg}-Track-and-Stop Strategy}\label{app:upper-bound}

The proof of~\autoref{lem:KL-concentation} uses the concentration bound~\autoref{thm:concentration-bound}, combined with the monotonicity of the Kullback-Leibler divergence rate~\autoref{cor:incr}.
\begin{proof}[Proof of~\autoref{lem:KL-concentation}]\hfill\break
We first note the following inclusion of events
\begin{align*}
& \bigcup_{t=1}^\infty \bigcup_{n=1}^t \left\{
N_a (t) \KL{\hat{\mu}_a (N_a (t))}{\mu_a} \ge \beta_{\alpha, \delta} (t)/2, ~ N_a (t) = n
\right\} \\
&\qquad \subseteq 
\bigcup_{t=1}^\infty \bigcup_{n=1}^t \left\{
n \KL{\hat{\mu}_a (n)}{\mu_a} \ge \beta_{\alpha, \delta} (t)/2
\right\} \\
&\qquad = \bigcup_{t=1}^\infty \left\{
t \KL{\hat{\mu}_a (t)}{\mu_a} \ge \beta_{\alpha, \delta} (t)/2
\right\},
\end{align*}
where the last equality follows because, by the monotonicity of $t \mapsto \beta_{\alpha, \delta} (t)/2$ we have that for each $n \in \Intpp{}$ and for each $t = n, n+1, \ldots$
\[
\left\{
n \KL{\hat{\mu}_a (n)}{\mu_a} \ge \beta_{\alpha, \delta} (t)/2
\right\}
\subseteq
\left\{
n \KL{\hat{\mu}_a (n)}{\mu_a} \ge \beta_{\alpha, \delta} (n)/2
\right\}.
\]
Combining this with a union bound we obtain
\begin{align*}
& \Prtheta \left(\exists t \in \Intpp{} : N_a (t) \KL{\hat{\mu}_a (N_a (t))}{\mu_a} \ge \beta_{\alpha, \delta} (t)/2\right) \\
&\qquad\le \Pr_{\theta_a} \left(\exists t \in \Intpp{} : t \KL{\hat{\mu}_a (t)}{\mu_a} \ge \beta_{\alpha, \delta} (t)/2\right) \\
&\qquad\le \sum_{t=1}^\infty \Pr_{\theta_a} \left(\KL{\hat{\mu}_a (t)}{\mu_a} \ge 
\frac{\beta_{\alpha, \delta} (t)}{2 t}\right).
\end{align*}
We focus on upper bounding
\[
\Pr_{\theta_a} \left(\KL{\hat{\mu}_a (t)}{\mu_a} \ge 
\frac{\beta_{\alpha, \delta} (t)}{2 t}, ~ \hat{\mu}_a (t) \ge \mu_a\right).
\]
Let $\mu_{a,t}$ be the unique (due to~\autoref{cor:incr}) solution (if no solution exists then the probability is already zero) of the equations
\[
\KL{\mu_{a,t}}{\mu_a} = \frac{\beta_{\alpha, \delta} (t)}{2 t}, \quad\text{and}
\quad  \mu_a \le \mu_{a,t} \le M.
\]
Then the combination of~\autoref{cor:incr} and~\autoref{thm:concentration-bound} gives
\[
\Pr_{\theta_a} \left(\KL{\hat{\mu}_a (t)}{\mu_a} \ge 
\frac{\beta_{\alpha, \delta} (t)}{2 t}, ~ \hat{\mu}_a (t) \ge \mu_a\right)
= \Pr_{\theta_a} \left(\hat{\mu}_a (t) \ge \mu_{a, t}\right) 
\le \frac{\delta}{D} \frac{1}{t^\alpha} C^2.
\]
We further upper bound the constant $c (P_{\mu_a})$ by $c (P)^2$
using~\autoref{lem:uniform}, in order to obtain a uniform upper bound for any Markovian arm coming from the family.

A similar bound holds true for
\[
\Pr_{\theta_a} \left(\KL{\hat{\mu}_a (t)}{\mu_a} \ge 
\frac{\beta_{\alpha, \delta} (t)}{2 t}, ~ \hat{\mu}_a (t) \le \mu_a\right).
\]
The conclusion follows by summing up over all $t$ and using the simple integral based estimate
\[
\sum_{t=1}^\infty \frac{1}{t^\alpha} \le \frac{\alpha}{1-\alpha}.
\]
\end{proof}
Embarking on the proof of the fact that the $(\alpha, \delta)$-Track-and-Stop strategy is $\delta$-PC we first show that the error probability is at most $\delta$ no matter the bandit model.
\begin{proposition}\label{prop:almost-PC}
Let $\btheta \in \bTheta$, $\delta \in (0, 1)$, and $\alpha > 1$.
Let $\calA_\delta$ be a sampling strategy that uses
an arbitrary sampling rule, the $(\alpha, \delta)$-Chernoff's stopping rule and the best sample mean decision rule. Then,
\[
\Prtheta (\tau_{\alpha, \delta} < \infty, \hat{a}_{\tau_{\alpha, \delta}} \neq a^* (\bmu)) \le \delta.
\]
\begin{proof}
The following lemma which is easy to check, and its proof is omitted, will be useful in our proof of~\autoref{prop:almost-PC}.
\begin{lemma}\label{lem:variational-formula}
The generalized Jensen-Shannon divergence
\[
I_a (\mu, \lambda) = a \KL{\mu}{a \mu + (1-a) \lambda} +
(1-a) \KL{\lambda}{a \mu + (1-a) \lambda},
~ \text{for} ~ a \in [0, 1]
\]
satisfies the following variational characterization
\[
I_a (\mu, \lambda) = \inf_{\mu' < \lambda'}\: \left\{a \KL{\mu}{\mu'} + (1-a) \KL{\lambda}{\lambda'}\right\}.
\]
\end{lemma}
If $\tau_{\alpha, \delta} < \infty$ and $\hat{a}_{\tau_{\alpha, \delta}} \neq a^* (\bmu)$,
then there $\exists t \in \Intpp{}$ and there $\exists a \neq a^* (\bmu)$
such that $Z_{a, a^* (\bmu)} (t) > \beta_{\alpha, \delta} (t)$. In this case we also have
\begin{align*}
\beta_{\alpha, \delta} (t)
&< Z_{a, a^* (\bmu)} (t) \\
&= N_a (t) \KL{\hat{\mu}_a (N_a (t))}{\hat{\mu}_{a,a^* (\bmu)} (N_a (t), N_{a^* (\bmu)} (t))} + \\
&\qquad N_{a^* (\bmu)} (t) \KL{\hat{\mu}_{a^* (\bmu)} (N_{a^* (\bmu)} (t))}
{\hat{\mu}_{a,a^* (\bmu)} (N_a (t), N_{a^* (\bmu)} (t))} \\
&= (N_a (t) + N_{a^* (\bmu)} (t))
I_{\frac{N_a (t)}{N_a (t) + N_{a^* (\bmu)} (t)}} (\hat{\mu}_a (N_a (t)), \hat{\mu}_{a^* (\bmu)} (N_{a^* (\bmu)} (t))) \\
&= \inf_{\mu_a' < \mu_a''}\:\left\{
N_a (t) \KL{\hat{\mu}_a (N_a (t))}{\mu_a'} + 
N_{a^* (\bmu)} (t) \KL{\hat{\mu}_{a^* (\bmu)} (N_{a^* (\bmu)} (t))}{\mu_a''}
\right\} \\
&\le N_a (t) \KL{\hat{\mu}_a (N_a (t))}{\mu_a} + 
N_{a^* (\bmu)} (t) \KL{\hat{\mu}_{a^* (\bmu)} (N_{a^* (\bmu)} (t))}{\mu_{a^* (\bmu)}},
\end{align*}
where the third equality follows from the variational formula for the generalized Jensen-Shannon divergence given in~\autoref{lem:variational-formula}, and the last inequality follows from the fact that $\mu_a < \mu_{a^* (\bmu)}$.

This in turn implies that,
$\beta_{\alpha, \delta} (t)/2 < N_a (t) \KL{\hat{\mu}_a (N_a (t))}{\mu_a}$, or
$\beta_{\alpha, \delta} (t)/2 < N_{a^* (\bmu)} (t) \KL{\hat{\mu}_{a^* (\bmu)} (N_{a^* (\bmu)} (t))}{\mu_{a^* (\bmu)}}$.
Therefore by union bounding over the $K$ arms we obtain
\begin{align*}
& \Prtheta (\tau_\delta < \infty, \hat{a}_{\tau_\delta} \neq a^* (\bmu)) \\
& \qquad \le \sum_{a=1}^K \Prtheta \left(\exists t \in \Intpp{} :
N_a (t) \KL{\hat{\mu}_a (N_a (t))}{\mu_a} \ge \beta_{\alpha, \delta} (t)/2 \right).
\end{align*}
The conclusion now follows by applying~\autoref{lem:KL-concentation}.
\end{proof}
\end{proposition}
\begin{proof}[Proof of~\autoref{prop:PC}]\hfill\break
Following the proof of Proposition 13 in~\cite{GK16},
and observing that in their proof they show that $\tau_{\alpha,\delta}$ is essentially bounded we obtain that 
\[
\Etheta [\tau_{\alpha, \delta}]  < \infty.
\]
This combined with~\autoref{prop:almost-PC} establishes that the
$(\alpha, \delta)$-Track-and-Stop strategy is $\delta$-PC.
\end{proof}
\begin{proof}[Proof of~\autoref{thm:upper-bound}]\hfill\break
Finally for the proof the sample complexity of the 
$(\alpha, \delta)$-Track-and-Stop strategy in~\autoref{thm:upper-bound} we follow the proof of Theorem 14 in~\cite{GK16}, where we substitute the usage of the law of large numbers with the law of large numbers for Markov chains, and in order to establish their Lemma 19 we use our concentration bound in~\autoref{thm:concentration-bound}.
\end{proof}
\end{appendices}

\end{document}